\documentclass[12pt,twoside]{amsart}
\usepackage[top=2.5cm, bottom=2.5cm, left=3cm, right=2.8cm]{geometry}
\usepackage{graphicx}
\usepackage{xcolor}
\usepackage{fancyhdr}
\usepackage[colorlinks=true, linkcolor=blue!60!black, citecolor=blue!60!black, urlcolor=blue!60!black]{hyperref}
\usepackage{ulem}

\newtheorem{theorem}{Theorem}[section]
\newtheorem{proposition}[theorem]{Proposition}
\newtheorem{lemma}[theorem]{Lemma}
\newtheorem{corollary}[theorem]{Corollary}

\theoremstyle{definition}
\newtheorem{definition}[theorem]{Definition}

\theoremstyle{remark}
\newtheorem{remark}[theorem]{Remark}

\theoremstyle{plain}
\newtheorem{maintheorem}{Theorem}

\begin{document}

\title
{On rigidity of Finsler manifolds without conjugate points and with constant $S$-curvature}

\author{Anthony J. García, José Barbosa Gomes, and Rafael O. Ruggiero}

\keywords{Finsler manifolds, conjugate points, rigidity,
S-curvature, Anosov flows}

\begin{abstract}
We study rigidity phenomena in closed Finsler manifolds without conjugate points under assumptions on the $S$-curvature. We prove that a closed $C^\omega$ Finsler manifold with constant $S$-curvature, continuous Green bundles, and admitting a hyperbolic closed geodesic must be Riemannian. In the $C^\infty$ setting, the same conclusion holds under the additional assumption that the geodesic flow is transitive. As a consequence, we obtain rigidity results for Finsler manifolds with uniform visibility universal covering. Our approach is based on the analysis of the Cartan vector field as a Jacobi field and its interaction with the geometry of Green bundles.
\end{abstract}

\maketitle 

\section{Introduction}

Finsler geometry is an important notion in classical mechanics, since the Hamiltonian flow in high energy levels of a Tonelli Hamiltonian can be modeled (up to reparametrization) by the geodesic flow of a Finsler metric. This fact is a well know generalization of the Maupertuis principle for Mechanical Lagrangians, and it is our main motivation to study Hamiltonian flows from a geometric perspective. A smooth manifold endowed with a Finsler metric is called a Finsler manifold. Riemannian manifolds are, of course, Finsler manifolds, but the Finsler world includes much more than Riemannian examples in view of our previous remark. 

One of the main problems in Riemann-Finsler geometry is the characterization of Finsler metrics that are Riemannian. Different notions of curvatures and tensors have been defined in Finsler theory to capture the non-Riemannian character of a Finsler manifold. Among them, the Cartan tensor plays a crucial role in the problem: it vanishes if and only if the Finsler metric is Riemannian. 

Special properties of the Cartan tensor under certain assumptions on the Finsler metric, the so-called k-basic or k-sectional assumption on the Finsler flag curvature, relate the Cartan tensor to the Jacobi fields of the Finsler Lagrangian. This establishes a relevant connection between the study of the Cartan tensor and the dynamics of its geodesic flow. Akbar-Zadeh, in a pioneering article \cite{AZ1988} takes advantage of this nice remark to show that hyperbolic dynamics in Finsler geometry is somehow related to rigidity features: compact Finsler manifolds of constant negative flag curvature must be Riemannian. The proof of this result strongly uses the fact that the Cartan tensor in compact Finsler manifolds of constant negative flag curvature satisfies the Jacobi equation. See \cite{BCS}.

Subsequent work by G. Paternain \cite{P1997}, Gomes-Ruggiero \cite{GR2013}, Foulon-Ruggiero \cite{FR2016}, extended Akbar-Zadeh's result to Finsler metrics in compact surfaces of genus greater than one under the k-basic assumption (meaning that the flag curvature does not depend on the flag). G. Paternain assumes in addition that the Finsler metric is analytic; in the other above mentioned references, the Finsler metric is assumed to have no conjugate points. All these results represent remarkable extensions of Akbar-Zadeh in the sense that the role of the hyperbolicity of the geodesic flow, granted by the negative curvature in Akbar-Zadeh's work, is replaced by positive topological entropy of the geodesic flow.


The goal of this article is to investigate the role of the so-called S-curvature (see Subsection \ref{S-curvature}) in Finsler manifolds as one of the quantifiers of the non-Riemannian character of a Finsler metric. The S-curvature was introduced by Z. Shen in \cite{Shen1997}, it measures the rate of change of the Finsler volume form with respect to the Busemann-Hausdorff volume form along geodesics. In particular, it vanishes in Riemannian metrics, as well as in Berwald manifolds
(\cite[Section 2]{Shen1997} or \cite{Shen2016}). Since there are many Berwald metrics that are non-Riemannian metrics, the vanishing of the S-curvature is not enough to show that a Finsler metric is Riemannian (See, for instance, \cite[Section 11.6]{BCS} or \cite[Section 2.3]{CS2005}). However, inspired by the work of Shen, we shall focus on the study of Finsler manifolds without conjugate points where the S-curvature has special properties. Let us explain in more detail. 

We say that an n-dimensional Finsler manifold $(M,F)$ has \textit{constant S-curvature} if $S=(n+1)cF$, for some constant $c$. There exist examples of Finsler metrics that are not Riemannian and have constant $S$-curvature (see \cite[Example 5.2.1]{CS2005}).

Shen obtained rigidity results for manifolds with constant $S$-curvature and negative flag curvature \cite{Shen2005}. More precisely:

\begin{theorem}[Shen, \cite{Shen2005}]
Let $(M, F)$ be a $C^\infty$ closed Finsler manifold with constant  $S$-curvature. If $F$ has negative flag curvature, $K < 0$, then it must be Riemannian.
\end{theorem}

In \cite[Lemma 4.1]{Mo2008}, it is shown that, for a compact Finsler manifold, constant $S$-curvature implies zero $S$-curvature (see also \cite[page 173]{Shen2016}).

The main idea of the proof of Shen's Theorem is the fact that, under the assumption on the S-curvature, the Cartan tensor satisfies the Jacobi equation of the Finsler Lagrangian along geodesics. Thus, the extension of Akbar-Zadeh's argument to Finsler manifolds with constant negative curvature is possible: the hyperbolicity of the dynamics implies that the Cartan tensor must vanish. Since the Cartan vector field is a bounded Jacobi field along the geodesics, it must vanish on hyperbolic sets; in particular, when the geodesic flow is Anosov (see Definition~\ref{hyperbolicity}), the manifold is Riemannian.

Closed Finsler manifolds $(M,F)$ with negative flag curvature have Anosov geodesic flow \cite{Foulon1992} and have no conjugate points. A natural question is then: what happens in Finsler manifolds without conjugate points? Is it possible to extend Shen's Theorem to this setting? 

A key tool in the study of the dynamics of geodesic flows in manifolds without conjugate points is the existence of two Lagrangian bundles over the unit tangent bundle that are invariant under the action of the differential of the geodesic flow, the so called Green bundles (see definition in Section \ref{preliminaries}). They were constructed for geodesics without conjugate points of Riemannian metrics by Green \cite{Green1958} and of Finsler metrics by Foulon \cite{Foulon1992}. 
Its construction for energy levels without conjugate points of convex Hamiltonians is due to Contreras and Iturriaga in \cite{CI-1999}. In examples such as Riemannian manifolds with nonpositive curvature, Green bundles are continuous, but in general they are not continuous \cite{BBB1987}. If the geodesic flow is Anosov, Green bundles coincide with the invariant bundles of the Anosov dynamics. 

Our main result is the following theorem.

\begin{maintheorem}
Let $(M, F)$ be a $C^\omega$ closed Finsler manifold without conjugate points, where Green bundles are continuous, and there exists a hyperbolic closed geodesic. If the S- curvature is constant, then $(M,F)$ is Riemannian.
\end{maintheorem}

The $C^\infty$ version follows by assuming the transitivity of the geodesic flow:

\begin{maintheorem}
Let $(M,F)$ be a $C^\infty$ closed Finsler manifold without conjugate points. Assume that the geodesic flow is transitive in $T_1M$, that the Green bundles are continuous, and that there exists a hyperbolic closed geodesic. If the $S$-curvature is constant, then $(M,F)$ is Riemannian.
\end{maintheorem}

The proof of Theorem A relies on the fact that, for manifolds with vanishing $S$-curvature, the Cartan vector field is a Jacobi field along every geodesic of $(M,F)$. This fact was used by Shen in \cite{Shen2005}. Roughly speaking, what we show is that if the set of points where the Green bundles are transversal is nonempty, then the Cartan vector field vanishes on such a set. This follows from a global property of Jacobi fields in manifolds without conjugate points: the divergence of radial Jacobi fields (see Lemma 2.13).

Visibility manifolds, introduced in the Riemannian setting in \cite{EO1973} and in the Finsler setting in \cite{CBR2020}, provide a natural framework that extends previous rigidity results by Shen and Akbar-Zadeh. Klingenberg proved that Riemannian manifolds with Anosov geodesic flow have uniform visibility universal covering \cite{Klingenberg1974}. These ideas are extended to the Finsler setting using the theory developed in \cite{CBR2020}. It is also shown there that the geodesic flow of compact Finsler manifolds without conjugate points, having a Finsler uniform visibility universal covering (see Definition 3.2 in \cite{CBR2020}), is transitive. Combining this with Theorem B, we obtain:

\begin{maintheorem}\label{visibility corollary}
Let $(M, F)$ be a compact $C^\infty$ Finsler manifold without conjugate points whose universal covering, endowed with the pullback metric, is a uniform visibility manifold, and assume that the Green bundles are continuous and that there exists a hyperbolic closed geodesic. If the $S$-curvature is constant, then $(M,F)$ is Riemannian.
\end{maintheorem}

Theorems A and B are proved in Section 3, after we give, in Section 2, the preliminary definitions and results that will be used in its proofs.

\section{Preliminaries} \label{preliminaries}

In this section, we give some basic definitions and results which are necessary in the following sections. For more details, we are 
referred to see \cite{BCS} and \cite{CS2005}. 

Let $M$ be an n-dimensional, $C^{\infty}$ manifold. We denote by $T_{p}M$ 
the tangent space at $p\in M$, and by $TM$ the tangent bundle of $M$.
In canonical coordinates, an element of $TM$ has the form $(x,y)$, where $x \in M$ and $y \in T_xM$. Let $TM_0$ be the complement of the zero section, that is, $TM_0= \{(x,y) \in TM; \, y \neq 0 \}$. A $C^{k}$
($k\geq 2$) \textit{Finsler structure} on $M$ is a continuous function
$F:TM \rightarrow [0,+ \infty)$ with the following properties: 
\begin{itemize}
\item $F$ is $C^k$ on $TM_0$; 
\item $ F(x,\lambda y)=\lambda F(x,y) \; \forall \; \lambda >0 \mbox{ and for every } (x,y) \in TM, $
that is, $F$ is positively homogeneous of degree one in $y$; 
\item The Hessian matrix of $F^2=F\cdot F$
$$g_{ij} := \frac{1}{2}\frac{\partial ^2}{\partial y^i \partial y^j}F^2 $$
is positive definite on $TM_0$.
\end{itemize}

The positive homogeneity of $F$ implies that $g_{(x,y)} = g_{(x,\lambda y)}$ for every $\lambda>0$. On the other hand, $g_{(x,y)}$ is independent of $y$ if and only if it is a Riemannian metric on $M$. With this in mind, it is natural to define the following non-Riemannian quantity, called the \textit{Cartan tensor}, by

$$
\textbf C_{y}(u,v,w)   : =\frac{1}{4}\frac{\partial^3}{\partial s \partial t \partial r}F^2(x,y+su+tv+rw)\big|_{r,s,t=0},
$$
where $u,v,w \in T_xM$, $y \in T_xM\setminus\{0\}$ and the \textit{mean Cartan torsion} by
$$
\textbf I_{y}(u):=tr_g(\textbf C_{(x,y)}(u,\cdot,\cdot))
$$
The \textit{Cartan vector field} $\vec{\mathbf{I}}$ is the $g$-dual vector of $\textbf{I}$ given by
\begin{align} \label{a}
    \mathbf {I}_{y}(u)=g_{(x,y)}(\vec {\mathbf{I}}({y}),u).
\end{align}

An interesting property of $\mathbf{C}$ is that
\begin{equation}
\textbf{C}_y(y,u,v) = \textbf{C}_y(u,y,v) = \textbf{C}_y(u,v,y) = 0.
\end{equation}\label{C es (-1)-homogeneo}
In fact, Euler’s identity implies that $(F^2)_{ij}$ is 0-homogeneous. In particular, this means that 
$(F^2)_{ijk}\, y^i = (F^2)_{ijk}\, y^j = (F^2)_{ijk}\, y^k = 0$. As a consequence, we obtain
\begin{equation}\label{b}
     \textbf{I}_y(y) = 0.
\end{equation}\label{ortogonal cartan}
Combining \ref{a} and \ref{b} we have that  $\vec {\mathbf{I}}_y$ is $g_y$-ortogonal to $y$.\\
The tensor $\textbf{I}$ allows us to derive rigidity conditions due to the following result.

\begin{theorem}[Deicke] \label{Deicke}
Let $(M,F)$ be a Finsler manifold. If $\textbf{I}=0$ then $g$ is Riemannian.
\end{theorem}
See (Theorem 14.4.1, \cite{BCS}) for a proof.

\subsection{S-curvature} \label{S-curvature}

In this subsection we follow \cite{CS2005} as main reference.

Let $F$ be a Finsler metric on a manifold $M$. For every tangent space $T_xM$, the Minkowski norm $F_x = F |_{T_xM}$ defines
\begin{equation*}
    \sigma_{F_x} = \frac{\text{vol}(B^n)}{\text{vol}(\hat{B}^n)},
\end{equation*}
where $B^n$ is the unit ball in $\mathbb{R}^n$, and $\hat{B}^n = \{(y^i) \in \mathbb{R}^n \mid F |_{T_xM} (y^i b_i) < 1\}$, where $\{b_i\}$ is a fixed basis of $T_xM$. 
Now, let
\begin{equation*}
    \tau_x := \ln \frac{\sqrt{\det(g_{ij}(y))}}{\sigma_{F_x}}.
\end{equation*}

This function $\tau_x = \tau_x(y)$ does not depend on the choice of $\{b_i\}$. The \textit{distortion} of $F$ is given by $\tau = \tau(x, y)$ on $TM \setminus \{0\}$.

The rate of change of the distortion along geodesics defines the \textit{S-curvature}:
\begin{equation*}
    S(x, y) := \frac{d}{dt} \Big|_{t=0} \tau(\sigma(t), \sigma'(t)),
\end{equation*}
where $\sigma(t)$ is a geodesic with $\sigma(0) = x$ and $\sigma'(0) = y$.

We say that an n-dimensional Finsler manifold $(M,F)$ has \textit{constant S-curvature} if there exists a constant $c$ such that
\begin{equation*}
    S = (n + 1)cF.
\end{equation*}

Every Riemannian metric has constant S-curvature.
Example 5.2.1 in \cite{CS2005} gives a Finsler metric with constant S-curvature, where 
$F$ is a Randers metric that is not Riemannian. Also, every Berwald metric has zero constant $S$-curvature (see \cite[Section 2]{Shen1997} or \cite[Proposition 4.3]{Shen2016}). When $(M,F)$ is a closed manifold, $S$ is constant if and only if $S=0$ (for details, see
{ \cite[Lemma 4.1]{Mo2008} or
\cite{Shen2016}, page 73).

We recall that there are several volume forms on a Finsler manifold, different from the Riemannian case. In the above definition of S-curvature, we are considering the Busemann-Hausdorff volume form.

\subsection{Chern-Rund connection, Jacobi fields and conjugate points}

The fundamental tensor $ g_{i j(x, y)} d x^{i} \otimes d x^{j} $ is very convenient to study Finsler geometry from a Riemannian viewpoint. In the next lemma we summarize some of its basic properties (see \cite{BCS},
for details).

\begin{lemma}\label{Covariantderivate}
    
 Let $(M, F)$ be a $C^{4}$ Finsler manifold, let $\sigma(t)$ be a $C^{\infty}$ curve, and $\sigma(t, u): \triangle=\{(t, u) ; 0 \leq t \leq r,-\varepsilon<u<\varepsilon\} \longrightarrow M$ be a $C^{2}$ variation of $\sigma(t, 0)=$ $\sigma(t)$ by $C^{\infty}$ curves. Then, in the tangent space $T_{\sigma(t, u)} M$ the inner product
$$ g_{T}:=g_{i j(\sigma(t, u), T(t, u))} d x^{i} \otimes d x^{j} $$
where $T=T(t, u):=\sigma_{*} \frac{\partial}{\partial t}=\frac{\partial \sigma}{\partial t}$, satisfies the following properties:

\begin{enumerate}
\item  $g_{T}(T, T)=F^{2}(T)$.

\item  $\sigma(t)$ is a Finslerian geodesic if and only if

$$ \frac{d}{d t} g_{T}(V, W)=g_{T}\left(D_{T} V, W\right)+g_{T}\left(V, D_{T} W\right) $$
where $V$ and $W$ are two arbitrary vector fields along $\sigma$. The operator $D_{T}=$ $\frac{d}{d t}$ is called covariant differentiation with reference vector $T$.

\item  In particular, Finslerian geodesics satisfy
$$ D_{T}\left[\frac{T}{F(T)}\right]=0 .$$
The constant speed Finslerian geodesics $F(v)=c$ are the solutions of
$$ D_{T} T=0 .$$
\item  Assume that $\sigma(t)$ is a unit speed geodesic. Then a Jacobi field $J(t)=$ $\frac{\partial \sigma}{\partial u}(t, 0)$ along $\sigma(t)$ satisfies
$$ D_{T} D_{T} J+R(J, T) T=0 $$
\end{enumerate}
where $R$ is the Jacobi tensor of the Finsler metric (We shall denote as usual $\left.J^{\prime \prime}=D_{T} D_{T} J, J^{\prime}=D_{T} J\right)$. 
\end{lemma}

Lemma \ref{Covariantderivate} reduces many Finsler problems concerning Jacobi fields to Riemannian ones. We shall often call the inner product $g_{T}$ the adapted Riemannian metric. Throughout the paper, all covariant differentiations will be carried out with reference vector $T$.

\begin{definition}We say that $q$ is \textit{conjugate} to $p$ along a geodesic $\sigma$ if there exists a nonzero Jacobi field $J$ along $\sigma$ which vanishes at $p$ and $q$. We say that $(M, F)$ \textit{has no conjugate points} if no geodesic has conjugate points.
\end{definition}

\begin{proposition}\label{constant S curvature}
    Let $(M,F)$ be a Finsler manifold, and let $\gamma: (a,b) \to M$ be a geodesic. If the $S$-curvature is constant, then the Cartan vector field $\vec{\mathbf{I}}$ along $\gamma$, defined by $\vec{\mathbf{I}}(t) = \vec{\mathbf{I}}(\gamma(t),\dot{\gamma}(t))$, is a Jacobi field.
\end{proposition}
See (Lemma 3.1, \cite{Shen2005}) for a proof.

\subsection{Sasaki type Metric}
As in the Riemannian case, we can consider a metric on the space $TM_0$ induced by $F$. Let $\{\frac{\partial}{\partial x^i}, \frac{\partial}{\partial y^j}\}$ be a local frame of $TTM_0$, and $\{dx^i, dy^j\}$ its dual in $T^*TM$. Consider
\begin{align*}
\frac{\delta}{\delta x^j} &:= \frac{\partial}{\partial x^j} - N^i_j \frac{\partial}{\partial y^i}, \\
\delta y_i &:= d y_i + N^i_j d x^j,
\end{align*}
where $N^j_i = \Gamma^j_{ik} y^k.$
The \textit{horizontal subbundle} ${H}$ of $TTM_0$ is the one generated by $\left\{\frac{\delta}{\delta x^i}\right\}$, and the \textit{vertical subbundle} ${V}$ is the one generated by $\left\{\frac{\partial}{\partial y^i}\right\}$.

\begin{definition}
The \textit{Sasaki (type) metric} on $TM_0$ is given in local coordinates by
    $$\tilde g _y(\cdot,\cdot):=g_{ij}(y)dx^i\otimes dx^j+g_{ij}(y)\frac{\delta y_i}{F}\otimes \frac{\delta y_j}{F}.$$
\end{definition} 

If $X$ is the geodesic vector field of $(M,F)$, we can associate to it the 1-form $\alpha_\theta : T_\theta TM _0\to \mathbb{R}$ given by
$$
\alpha_\theta(\xi) := \tilde g_{X(\theta)}(X(\theta),\xi).
$$

Consider on $T_1M = \{\theta \in TM : F(\theta) = 1\}$ the subset $N(\theta) = \ker(\alpha_\theta) \subseteq T_\theta T_1M$ and define the fibers:
$$
\mathcal{H}_\theta = H_\theta \cap N(\theta), \quad \mathcal{V}_\theta = V_\theta \cap N(\theta).
$$

The \textit{Liouville measure} on the unit tangent bundle $T_1M$ is then the volume form
\[
\mu:=\frac{1}{(n-1)!}\,\alpha \wedge (d\alpha)^{\,n-1},
\]
where $n = \dim M$. A classical fact is that the geodesic flow $\phi_t: T_1M \to T_1M$ preserves this measure (Proposition 4.7, \cite{Shen2016}).

Given $\theta=(x,v) \in TM_0$, we will construct an isomorphism between $T_{(x,v)} TM_0$ and $T_xM \oplus T_xM$, which allows us to translate problems related to the dynamics of the geodesic flow into the study of the growth of the norm of Jacobi fields. 

\begin{definition}
   
Let $Z \in T_{(x,v)}(TM_0)$. Choose a smooth curve $c: (-\epsilon,\epsilon) \to M$ with $c(0) = x$
and a vector field $V$ along $c$ with $V(0) = v$, such that the lifted curve
$\alpha(t) = (c(t), V(t)) \in TM$ satisfies $\alpha'(0) = Z$. We define the \textit{vertical operator} $K$ as:

$$K_\theta(Z) := \left({D_{c'}^{c'}}V\right)(0).$$
\end{definition}
\begin{lemma}
Let $\theta=(x,v) \in TM_0$. The following properties are satisfied:

\begin{enumerate}
    \item $\ker K_\theta = \mathcal{H}_\theta$
    \item The maps 
          $$d\pi_\theta|_{\mathcal{H}_\theta}: \mathcal{H}_\theta \to T_x M \quad \text{and} \quad 
            K_\theta|_{\mathcal{V}_\theta}: \mathcal{V}_\theta \to T_x M$$
          are linear isomorphisms.
\end{enumerate}

\end{lemma}

\begin{proposition}
The correspondence
    \begin{equation}\label{correspondence}
    \begin{aligned}
       j_\theta: T_\theta TM_0 &\to T_xM\oplus T_xM\\
        \xi &\mapsto ( d\pi_\theta(\xi),K_\theta(\xi))
    \end{aligned}
    \end{equation}
is a linear isomorphism.
\end{proposition}

\begin{proposition}\label{correspondencia jacobi dif}
Let $(M,F)$ be a Finsler manifold, and let $\phi_t : TM_0 \rightarrow TM_0$ denote its geodesic flow. 
Let $(x,v) \in TM_0$ and $\xi \in T_{\theta}TM_0$. Then there exists an isomorphism such that
\begin{equation}
    d_{\theta}\phi_t\, \xi \simeq \big( J_\xi(t),  J'_\xi(t) \big)
\end{equation}
where $J_\xi$ is the Jacobi vector field along $\gamma_{(x,v)}$ with initial condition 
$\xi \simeq \big( J_\xi(0),  J'_\xi(0) \big)$, and $\simeq$ denotes the correspondence given by $j_\theta.$
\end{proposition}
See Appendix of \cite{HS2013} for a proof.

\begin{corollary}\label{diferential-jacobi norm}
With the hypotheses of the theorem above, for every $\theta \in TM_0$,
$$
\|d_\theta\phi_t\xi\|^2_{\hat{g}} = \|J_\xi(t)\|_g^2 + \|J'_\xi(t)\|^2_g \quad \forall t \in \mathbb{R}.
$$
\end{corollary}

\subsection{Green bundles}\label{Green section}

Let $\gamma_\theta:\mathbb{R} \rightarrow M$ be a geodesic and consider $e_2, \ldots, e_n$ orthonormal parallel to $T := \gamma_\theta'$. The Jacobi field $J_\theta(t)$ is a linear combination of the parallel vector fields:  
$$J_\theta(t) = \sum_{i=1}^{n-1} g_T(J_\theta(t), e_i(t)) e_i(t).$$  
Given any collection $J_{ j}(t), \, j = 1, 2, \ldots, n-1$ of Jacobi fields perpendicular to $\gamma_\theta'(t)$, since $\gamma_\theta$ is a geodesic, the compatibility condition holds (Lemma \ref{Covariantderivate}-2), and therefore:  

$$J_\theta'(t) = \sum_{i=1}^{n-1} g_T(J_\theta'(t), e_i(t)) e_i(t),$$  
and the $(n-1) \times (n-1)$ matrix $\mathcal{J}_\theta(t)$, whose entries are:  
$$\mathcal{J}_{ ij}(t) = g_T(J_{ j}(t), e_i(t)),$$  
is a solution of the matrix Jacobi equation:  
\begin{equation} \label{Jacobi}
    \mathcal{J}_\theta''(t) + K_\theta(t) \mathcal{J}_\theta(t) = 0 \tag{J},
    \end{equation}
where $\mathcal{J}_\theta'(t)$ is the componentwise derivative of $\mathcal{J}_\theta(t)$ and $K(t)$ is a symmetric matrix whose entries are the sectional curvatures:  
$$K_{ ij} = g_T(R^T(T, e_j) T, e_i).$$  

In particular, if a matrix $\mathcal{J}_\theta$ is invertible, the new matrix $U_{\mathcal{J}} = \mathcal{J}_\theta' \mathcal{J}_\theta^{-1}$ satisfies the Riccati equation:  

\begin{equation}\label{Ricatti}
U_\theta'(t) + U_\theta^2(t) + K_\theta(t) = 0. \tag{R}
\end{equation}

Let $\mathcal{J}_{\theta, r}(t)$ be the Jacobi field whose boundary conditions are $\mathcal{J}_{\theta, r}(0) = \text{Id}$ and $\mathcal{J}_{\theta, r}(r) = 0$.  
If $(M, F)$ has no conjugate points along $\gamma_\theta$, we have that $\mathcal{J}_{\theta, r}$ exists and is unique. The following result is the Finsler version of a result by Green, and its proof was given in \cite{Foulon1992}.  

\begin{lemma}  \label{Matrices de Green}
Let $(M, F)$ be a compact Finsler manifold without conjugate points. Then the limits  
$$\lim_{r \to +\infty} \mathcal{J}_{\theta, r}(t) = \mathcal{J}^s_\theta(t) \quad \text{and} \quad \lim_{r \to -\infty} \mathcal{J}_{\theta, r}(t) = \mathcal{J}^u_\theta(t)$$  
exist for every $\theta \in TM$ and every $t \in \mathbb{R}$. These limits correspond to Jacobi matrices that are invertible for every $t \in \mathbb{R}$.  
\end{lemma}  




\begin{proposition}\label{Green properties}\cite{{Foulon1992}}
Let $(M, F)$ be a Finsler manifold without conjugate points. Then there exists, for each $\theta \in T_1M$, two invariant, $n-1$ dimensional subspaces $G^s(\theta)$ and $G^u(\theta)$ of $N(\theta)$, defined by
\begin{align*}
    G^s(\theta) &= \lim_{t \to +\infty} d\phi_{-t}(\mathcal{V}_{\phi_t(\theta)}), \\
    G^u(\theta) &= \lim_{t \to -\infty} d\phi_{-t}(\mathcal{V}_{\phi_t(\theta)}),
\end{align*}
where $D\phi_t$ is the differential of the geodesic flow, and $V_\theta$ is the vertical subspace at $\theta$. The distributions $G^s(\theta)$ and $G^u(\theta)$ satisfy the following properties:

\begin{enumerate}
    \item They are measurable, transverse to the vertical subbundle $V$, and transverse to the geodesic vector field.
    \item There exist linear operators $U^s_\theta : \mathcal H_\theta \to \mathcal V_\theta$, $U^u_\theta : \mathcal H_\theta \to \mathcal V_\theta$, such that $G^s(\theta)$, $G^u(\theta)$ are, respectively, the graphs of $U^s_\theta$, $U^u_\theta$. Such operators give rise to solutions of a well-known Riccati equation (see Proposition 2.2 in \cite{Foulon1992}, for instance), and there exists $L > 0$ such that $ \| U^s_\theta \|_\infty \leq L, \quad \| U^u_\theta \|_\infty \leq L, \quad \text{for every } \theta$.
    \item For every $\theta \in T_1M$ and $\xi \in G^s(\theta)\cup G^u(\theta)$ we have
$$\tilde \|J'_\xi(t)\|_g\leq L \|J_\xi(t)\|_g.$$
\end{enumerate}
\end{proposition}

The bundles $G^s$ and $G^u$ above are called the stable and unstable \textit{Green bundles}, respectively.

\begin{corollary}\label{Green-Jacobi inequality} For every $\theta \in T_1M$ and $\xi \in G^s(\theta)\cup G^u(\theta)$ we have
$$\|d_\theta\phi_t \xi\|_{\hat g}\leq\sqrt{{1+L^2}}\|J_\xi(t)\|_g$$
    
\end{corollary}
\begin{proof}
    It follows from Corollary \ref{diferential-jacobi norm}
    and Proposition \ref{Green properties}.
\end{proof}

Finsler manifolds without conjugate points enjoy the property that any Jacobi field vanishing at one point (called a radial Jacobi field) must be divergent. This was established in the Riemannian setting by Eberlein \cite{Eberlein1972}, and later extended to convex Lagrangians by Contreras and Iturriaga \cite{CI-1999}. The next lemma is a consequence of the divergence of such radial Jacobi fields, and the proof follows exactly the same argument as in the Riemannian case (see Lemma 3.3, \cite{Ruggiero2007}).
\begin{lemma}\label{stable-bounded} Let $(M,F)$ be a compact Finsler manifold without conjugate points. Supose that there exist a geodesic $\gamma_\theta$ and a perpendicular Jacobi field $J$ along $\gamma_\theta$ such that $\|J(t)\|\leq C$ for every $t\geq0$. Then $J$ is an stable Jacobi field, i.e.,  $(J(0),J'(0))\in G^s(\theta).$ Analogously, if $\|J(t)\|\leq C$ for every $t\leq0$, then is an unstable Jacobi field, i.e., $(J(0),J'(0))\in G^u(\theta).$
\end{lemma}

\begin{definition} \label{hyperbolicity}
A subset $S \subset T_1M$ that is invariant under the geodesic flow of $(M, F)$ 
is called a \textit{hyperbolic set} if there exist constants $C > 0$ and $\lambda > 0$ such that for every 
$\xi \in S$, we have $d\phi_t$-invariant subspaces $E^s(\xi)$ and $E^u(\xi)$ with the following properties:
\begin{enumerate}
    \item $E^s(\xi) \oplus E^u(\xi) \oplus X(\xi) = T_\xi T_1M$, where $X(\xi)$ is the subspace tangent to the flow direction.
    \item $\|d_\xi\phi_t|_{E^s(\xi)}\| \leq Ce^{-\lambda t}$ for every $t \geq 0$.
    \item $\|d_\xi\phi_t|_{E^u(\xi)}\| \leq Ce^{\lambda t}$ for every $t \leq 0$.
\end{enumerate}
\end{definition}

When the subset $S$ coincides with $T_1M$, the geodesic flow is called \textit{Anosov}.

\begin{remark}
   When the geodesic flow $\phi_t$ of $(M,F)$ is Anosov, then $G^s = E^s$ and $G^u = E^u$.
\end{remark}

\section{Main Results}
Consider the set 
$$
E := \{\theta \in T_1M : G^s(\theta) \cap G^u(\theta) = \{0\}\}
$$
of points where the stable and unstable Green bundles intersect trivially. The set $E$ is invariant under the geodesic flow $\phi_t$.


\begin{lemma}\label{Mainlemma} Let $(M, F)$ be a $C^\infty$ closed Finsler manifold without conjugate points, and let $E \neq \emptyset$. If the $S$-curvature is zero, the Cartan vector field $\vec{\mathbf{I}}$ vanishes on $E$.
\end{lemma}

\begin{proof}
Let $\theta=(x,v)\in E$, and let $\gamma_\theta$ be the corresponding geodesic. Since the $S$-curvature is zero, Proposition \ref{constant S curvature} implies that the vector field ${\vec{\mathbf{I}}}_\theta(t)= \vec{\mathbf{I}}(\gamma_\theta(t),\gamma'_\theta(t))$ along $\gamma_\theta$ is a Jacobi field. Let $\xi \in T(T_1M)_\theta$ be such that $\xi\simeq(\vec{\mathbf{I}}_\theta(0),\vec{\mathbf{I}}_\theta'(0))$ via the correspondence (\ref{correspondence}). Since $g_v(\vec{\mathbf{I}}(v),v)=0$ by (\ref{ortogonal cartan}), it follows that $\xi\in N(\theta).$ Since $\vec{\mathbf{I}}$ is bounded on $T_1M$, then $\vec{\mathbf{I}}_\theta(t)$ is bounded for every $t\in \mathbb R$. Hence Lemma \ref{stable-bounded} implies that it is both a stable and unstable Jacobi field, so $\xi \in G^s(\theta)\cap G^u(\theta)$, and since $\theta \in E$, it follows that $\xi = 0$, which implies $\vec{\mathbf{I}}(\theta)=0$. Therefore, the Cartan vector field vanishes on $E$.
\end{proof}
\begin{proof}[Proof of Theorem A] Since the Green bundles are continuous, it follows that $E$ is an open subset of $T_1M$. On the other hand, if there exists a hyperbolic periodic geodesic, then $E\neq \{0\}$ (Proposition B, \cite{CI-1999}). Since $\vec{\mathbf{I}}$ vanishes on $E$ by Lemma~\ref{Mainlemma}, it follows from Łojasiewicz's Theorem \cite{Krantz2002} on the structure of the zero set of real-analytic functions that $\vec{\mathbf{I}}$ vanishes identically on $T_1M$. Applying Deicke's Theorem~\ref{Deicke} completes the proof.

\end{proof}

\begin{proof}[Proof of Theorem B]
Since $(M,F)$ has a transitive geodesic flow, there exists a dense geodesic $\gamma_\theta$ in $T_1M$. Since $E \neq \emptyset$ is invariant under the geodesic flow and open in $T_1M$, we conclude that $\gamma_\theta \subset E$, and hence $\overline{E} = T_1M$. It follows from Lemma~\ref{Mainlemma} that $\vec{\mathbf{I}}$ vanishes on $T_1M$, and the result follows by Deicke's Theorem~\ref{Deicke}.
\end{proof}


\begin{thebibliography}{99} 

\bibitem{AZ1988} Akbar-Zadeh, H. \textit{Sur les espaces de Finsler á courbures sectionnelles constantes.} Bulletins de l'Académie Royale de Belgique \textbf{74} 
(1988), 281--322.

\bibitem{BCS} Bao, D., Chern, S.-S., Shen, Z. \textit{An Introduction to Riemann-Finsler Geometry.} 
Springer-Verlag, 
2000.

\bibitem{BBB1987} Ballmann, W.; Brin, M.; Burns, K. \textit{On surfaces with no conjugate points.} J. Differential Geometry \textbf{25} (1987), 249--273.

\bibitem{CS2005}S. S. Chern and Z. Shen. Riemann-Finsler geometry. World Scientific, 2005.

\bibitem{CBR2020} Chimenton, A. G.; Gomes, J. B.; Ruggiero, R. O. \textit{Gromov-hyperbolicity and transitivity of geodesic flows in n-dimensional Finsler manifolds.} Differential Geometry and its Applications \textbf{68} (2020), 101588.

\bibitem{CI-1999} Contreras, G.; Iturriaga, R. \textit{Convex Hamiltonians without conjugate points.} Ergodic Theory and Dynamical Systems \textbf{19} (1999), 901--952. 


\bibitem{Eberlein1972} Eberlein, Patrick. \textit{Geodesic flow in certain manifolds without conjugate points.} Transactions of the American Mathematical Society \textbf{167} (1972), 151--170.

\bibitem{EO1973} Eberlein, P.; O’Neill, B. \textit{Visibility manifolds.} Pacific Journal of Mathematics \textbf{46} 
(1973), 45--109.

\bibitem{Foulon1992} Foulon, P. \textit{Estimation de l'entropie des systèmes lagrangiens sans points conjugués.} Annales de l'IHP Physique théorique 
\textbf{57} (1992), No. 2, 117--146).

\bibitem{FR2016} Foulon, P.; Ruggiero, R.: \textit{A first integral for $C^\infty$, k-basic Finsler surfaces and applications to rigidity.} Proceedings of the American Mathematical Society \textbf{144} 
(2016), 3847--3858.

\bibitem{GR2013} Gomes, J. B.; Ruggiero, R. O.: \textit{On Finsler surfaces without conjugate points.} Ergodic Theory and Dynamical Systems \textbf{33} 
(2013), 455--474.

\bibitem{Green1958} Green, L. W. \textit{A theorem of E. Hopf.} Michigan Math. J. \textbf{5} (1958), 31--34. 

\bibitem{HS2013} Hryniewicz, U. L.; Salomão, P. A. \textit{Global properties of tight Reeb flows with applications to Finsler geodesic flows on $S^2$.} In: Mathematical Proceedings of the Cambridge Philosophical Society, Vol. \textbf{154}, No. 1, pp. 1--27. Cambridge University Press (2013). 



\bibitem{Klingenberg1974} Klingenberg, Wilhelm. \textit{Riemannian manifolds with geodesic flow of Anosov type.} Annals of Mathematics \textbf{99}, 1 (1974), 1--13.

\bibitem{Krantz2002}Krantz, Steven G., and Harold R. Parks. A primer of real analytic functions. Springer Science and Business Media, 2002.


\bibitem{Mo2008} Mo, X. \textit{A global classification result for Randers metrics of scalar curvature on closed manifolds.} Nonlinear Analysis \textbf{69} (2008), 2996--3004.


\bibitem{P1997} Paternain, G. P. \textit{Finsler structures on surfaces with negative Euler characteristic.} Houston Journal of Mathematics \textbf{23} (1997), 421--426.

\bibitem{Ruggiero2007} Ruggiero, Rafael O. \textit{Dynamics and global geometry of manifolds without conjugate points.} Ensaios Matemáticos \textbf{12} (2007), 1--181.


\bibitem{Shen1997} Shen, Zhongmin. \textit{Volume Comparison and Its Applications in Riemann-Finsler Geometry.} Advances in Mathematics \textbf{128} (1997), 306--328.

\bibitem{Shen2005} Shen, Zhongmin. \textit{Finsler manifolds with nonpositive flag curvature and constant S-curvature.} Mathematische Zeitschrift \textbf{249}, 3 (2005), 625--639.

\bibitem{Shen2016}Shen, Y. B.; Shen, Z. Introduction to modern Finsler geometry. World Scientific Publishing Company, 2016.

\end{thebibliography}
\end{document}